\documentclass[reqno, 11pt]{amsart}
\usepackage{url}
\usepackage[colorlinks, linkcolor=blue, citecolor=blue, urlcolor=blue]{hyperref}
\usepackage[usenames,dvipsnames]{xcolor}

\usepackage{verbatim}
\usepackage{float}
\usepackage{framed}
\usepackage{mdframed}
\usepackage{stmaryrd}

\binoppenalty=\maxdimen
\relpenalty=\maxdimen

\usepackage{helvet}

\usepackage{amsmath}
\usepackage{amsthm}
\usepackage{amssymb}
\usepackage{mathrsfs}
\usepackage{amsxtra}
\usepackage{pifont}
\usepackage{amsbsy}

\usepackage{pb-diagram}
\usepackage{lamsarrow}
\usepackage{pb-lams}

\usepackage{tikz}
\makeatletter
\newcommand\mathcircled[1]{%
  \mathpalette\@mathcircled{#1}%
}
\newcommand\@mathcircled[2]{%
  \tikz[baseline=(math.base)] \node[draw,circle,inner sep=1pt] (math) {$\m@th#1#2$};%
}
\makeatother

\numberwithin{equation}{section}
\numberwithin{table}{section}
\numberwithin{figure}{section}

\theoremstyle{plain}
\newtheorem{theorem}{Theorem}[section]
\newtheorem{lemma}[theorem]{Lemma}

\newtheorem{proposition}[theorem]{Proposition}

\newtheorem{corollary}[theorem]{Corollary}

\theoremstyle{definition}

\newtheorem{remark}[theorem]{Remark}

\numberwithin{theorem}{section}

\usepackage{bbm}
\usepackage{euscript}

\newcommand{\ga}{\mathfrak{a}}

\newskip\aline \newskip\halfaline
\def\skipaline{\vskip\aline}

\def\qedbox{$\rlap{$\sqcap$}\sqcup$}
\def\qed{\nobreak\hfill\penalty250 \hbox{}\nobreak\hfill\qedbox\skipaline}

\newcommand{\cond}{\,\Vert\,}

\newcommand*\xbar[1]{%
   \hbox{%
     \vbox{%
       \hrule height 0.5pt 
       \kern0.5ex
       \hbox{%
         \kern-0.1em
         \ensuremath{#1}%
         \kern-0.1em
       }%
     }%
   }%
}

\DeclareFontFamily{OT1}{pzc}{}
\DeclareFontShape{OT1}{pzc}{m}{it}{<-> s * [1.10] pzcmi7t}{}
\DeclareMathAlphabet{\mathpzc}{OT1}{pzc}{m}{it}

\newcommand{\one}{{\mathbbm{1}}}

\newcommand\bE{{\mathbb E}}

\newcommand{\bN}{{{\mathbb N}}}

\newcommand{\bP}{{{\mathbb P}}}

\newcommand\bR{{\mathbb R}}

\newcommand{\bV}{{{\mathbb V}}}

\newcommand\bZ{{\mathbb Z}}

\newcommand{\ba}{{\boldsymbol{a}}}

\newcommand{\be}{{\boldsymbol{e}}}

\newcommand{\ii}{\boldsymbol{i}}

\newcommand{\bm}{\boldsymbol{m}}

\newcommand{\bu}{{\boldsymbol{u}}}
\newcommand{\bv}{{\boldsymbol{v}}}

\newcommand{\bx}{{\boldsymbol{x}}}
\newcommand{\by}{{\boldsymbol{y}}}
\newcommand{\bz}{{\boldsymbol{z}}}

\newcommand{\bsI}{\boldsymbol{I}}

\newcommand{\bsK}{{\boldsymbol{K}}}

\newcommand{\bsV}{{\boldsymbol{V}}}

\newcommand{\bsX}{{\boldsymbol{X}}}
\newcommand{\bsY}{{\boldsymbol{Y}}}

\newcommand{\fC}{\mathfrak{C}}

\newcommand{\fM}{\mathfrak{M}}

\newcommand{\bgamma}{{\boldsymbol{\gamma}}}

\newcommand{\bGamma}{\boldsymbol{\Gamma}}

\newcommand{\blam}{{\boldsymbol{\lambda}}}

\newcommand{\bnu}{{\boldsymbol{\nu}}}

\newcommand{\vfi}{{\varphi}}

\newcommand{\eB}{\EuScript{B}}

\newcommand{\eI}{{\EuScript{I}}}

\newcommand{\eK}{\EuScript{K}}
\newcommand{\eL}{\EuScript{L}}
\newcommand{\eM}{\EuScript{M}}
\newcommand{\eN}{\EuScript{N}}

\newcommand{\eS}{\EuScript{S}}

\newcommand{\eV}{\EuScript{V}}

\newcommand{\eX}{{\mathfrak{X}}}

\newcommand{\mV}{\mathscr{V}}

\DeclareMathOperator{\tr}{{\rm tr}}

\DeclareMathOperator{\supp}{{\rm supp}}

\DeclareMathOperator{\vol}{vol}

\DeclareMathOperator{\Sym}{\mathbf{Sym}}

\DeclareMathOperator{\Hess}{Hess}

\DeclareMathOperator{\var}{Var}
\DeclareMathOperator{\cov}{Cov}

\DeclareMathOperator{\Var}{Var}

\DeclareMathOperator{\Prob}{Prob}

\DeclareMathOperator{\Meas}{Meas}

\newcommand{\ra}{\rightarrow}

\newcommand{\Ra}{\Rightarrow}

\newcommand{\Lra}{{\longrightarrow}}

\newcommand{\Llra}{{\Longleftrightarrow}}

\newcommand{\lan}{\langle}
\newcommand{\ran}{\rangle}

\newcommand{\rb}{\,\big]}
\newcommand{\lb}{\big[\,}

\newcommand{\rp}{\,\big)}
\newcommand{\lp}{\big(\,}
\newcommand{\lv}{\big\vert\,}
\newcommand{\rv}{\,\big\vert}
\newcommand{\Lv}{\Big\vert\,}
\newcommand{\Rv}{\,\Big\vert}

\newcommand{\lV}{\big\Vert\,}
\newcommand{\rV}{\,\big\Vert}

\newcommand{\Rp}{\,\Big)}
\newcommand{\Lp}{\Big(\,}

\def\inpr{\mathbin{\hbox to 6pt{\vrule height0.4pt width5pt depth0pt \kern-.4pt \vrule height6pt width0.4pt depth0pt\hss}}}

\newcommand{\pa}{\partial}

\newcommand{\hPhi}{{\widehat{\Phi}}}
\newcommand{\hphi}{{\widehat{\Phi}}}

\newcommand{\hh}{\widehat{H}}

\newcommand{\hrho}{\widehat{\rho}}

\usepackage{mathtools}
\def\rbinom#1#2{\ensuremath{\left(\kern-.3em\left(\genfrac{}{}{0pt}{}{#1}{#2}\right)\kern-.3em\right)}}

\newcommand{\as}{\mathrm{a.s.}}
\newcommand{\iid}{\mathrm{i.i.d.}}

\newcommand{\cpt}{{\mathrm{cpt}}}

\newcommand{\op}{{\mathrm{op}}}

\newcommand{\tphi}{{\widetilde{\Phi}}}
\newcommand{\trho}{\widetilde{\rho}}
\newcommand{\tH}{\widetilde{H}}

\begin{document}

\title[The distribution of critical points]{A  law of large numbers concerning the distribution of critical points of isotropic Gaussian functions} 

\author{Liviu I. Nicolaescu}
\address{L. Nicolaescu: Department of Mathematics, University of Notre Dame, Notre Dame, IN 46556-4618.}
\email{lnicolae@nd.edu}
\thanks{}

\begin{abstract}   We investigate the distribution of critical points of  certain isotropic random functions $\Phi$ on $\mathbb{R}^m$. We show that  the distribution of critical points of $\Phi(Rx)$, suitably normalized,  converges   a.s.  and $L^2$ to the Lebesgue measure as $R\to\infty$.  We achieve this by producing  precise asymptotics of the second moments of these distributions as $R\to\infty$.  \end{abstract}

\date{ Last revised \today.}
\keywords{isotropic Gaussian functions,  critical points,  Kac-Rice formula,  law of large numbers}

\address{Department of Mathematics, University of Notre Dame, Notre Dame, IN 46556-4618.}
\email{nicolaescu.1@nd.edu}
\urladdr{\url{http://www.nd.edu/~lnicolae/}}

\maketitle

\tableofcontents

\section{Introduction}

Denote by  $\Meas(\bR^m)$ the space of finite Borel measures on $\bR^m$. Suppose that $\ga:\bR\to\bR$ is an even Schwartz function  such that $\ga(0)=1$. We will refer to such functions as \emph{amplitudes}. Consider the  measure $\mu_\ga\in \Meas(\bR^m)$
 \[
 \mu_\ga\lb d\xi\rb =\frac{1}{(2\pi)^m}w_{\ga, m}\lp \xi\rp\blam\lb d\xi\rb,\;\; w_{\ga, m}\lp \xi\rp=\ga \lp |\xi|\rp^2,
 \]
where $\blam$ denotes the Lebesgue measure on $\bR^m$.  

The characteristic function  of  $\mu_\ga$ is the nonnegative definite function
 \begin{equation}\label{iso_sch_1}
\bsK_\ga\lp \bx\rp =\int_{\bR^m} e^{\ii \lan \xi,\bx\ran}\mu_\ga\lb d\xi\rb= \frac{1}{(2\pi)^m}\int_{\bR^m} e^{\ii \lan \xi,\bx\ran}\ga\lp |\xi|\rp^2\blam\lb  d\xi\rb.
 \end{equation}
 Clearly $\bsK_\ga(\bx)$ is an  $O(n)$-invariant, real valued Schwartz function. Then $\bsK_\ga\lp \bx-\by\rp$  is the covariance  kernel of a  real valued, smooth isotropic Gaussian function $\Phi=\Phi_\ga$ on $\bR^m$ with spectral measure $\mu_\ga$. 
 
 For $R>0$ we set 
 \[
 \ga_R(t):= \ga(t/R\rp,\;\;\forall t\in\bR.
 \]
Consider the  finite Borel measure $\mu_\ga^R\in \Meas(\bR^m)$
 \[
 \mu_\ga^R\lb d\xi\rb =\frac{1}{(2\pi)^m}w_{\ba_R, m}\lp \xi\rp\blam\lb d\xi\rb= \frac{1}{(2\pi)^m}\ga\lp \vert \xi\vert/R\rp^2 \blam\lb d\xi\rb.
  \]
Its characteristic function is the nonnegative definite function
 \begin{equation}\label{iso_sch}\begin{split}
 \bsK_\ga^R\lp \bx\rp = \frac{1}{(2\pi)^m}\int_{\bR^m} e^{\ii \lan \xi,\bx\ran}\ga\lp |\xi|/R\rp^2 d\xi =R^m\bsK_\ga\lp R\bx\rp.
 \end{split}
 \end{equation}
We deduce that  $\bsK_\ga^R\lp \bx-\by\rp$  is the covariance kernel of the Gaussian function 
 \[
 \Phi^R_\ga(\bx):=R^{m/2}\Phi_\ga\lp R\bx\rp.
 \]
 The Fourier inversion formula shows that
  \[
 \int_{\bR^m}\bsK_\ga(\bx)d\bx= \ga(0)^2=1.
 \]
 Since $K_\ga(\bx)$ is  $O(m)$-invariant  and smooth,  it has the form $\Psi\lp |\bx|^2\rp $ for some function $\Psi:[0,\infty)\to\bR$. According to Schoenberg's characterization theorem\ \cite[Thm.7.13]{Wen}, the function $\Psi$ must be completely monotone. In particular, $\Psi$  is non-increasing, nonnegative and convex, \cite[Lemma.7.3]{Wen}.  This implies that  the probability  measures $ \bsK^R_\ga\lp \bx\rp dx$ converge weakly to the  Dirac measure  $ \delta_0$. For example, if $\ga(t)=e^{-t^2/4}$, then  
 \[
 \bsK_\ga^R(\bx)=\frac{1}{(2\pi\hbar^{2})^{m/2}}e^{-\frac{\vert \bx\vert^2}{2\hbar^2} },\;\;\hbar=R^{-1}
 \]
 is the probability density the Gaussian measure on $\bR^m$ with variance $\hbar^2\one_m$.
 
 To use a terminology favored by  physicists, 
 \[
  \bsK_\ga^R\lp \bx-\by\rp\to\delta(\bx-\by).
  \]
  In other words, as $R \nearrow \infty$,  the  Gaussian random function $\Phi_\ga^R$   converges  in  distribution  to a Gaussian random  ``function'' $\Phi^\infty$ whose covariance kernel is $\bsK^\infty(\bx-\by)=\delta(\bx-\by)$.  This is  the  Gaussian white noise determined by the Lebesgue measure, \cite{GeVil_4}.

 For every $R>0$ the random function $\Phi^R_\ga$  is  a.s. Morse and there is an associated critical random measure 
 \[
 \fC^R_\ga:=\sum_{\nabla\Phi^R_\ga(\bx)=0}\delta_\bx.
 \]
 Thus, for any Borel subset $S\subset \bR^m$, $\fC^R_\ga[S]$ is the number of critical points of $\Phi^R_\ga$ in $S$.  More generally, if $f: \bR^m\to\bR$ is a continuous compactly supported function we set
 \[
  \fC^R_\ga[f]=\int_{\bR^m} f(\bx)\fC^R_\ga\lb d\bx\rb=\sum_{\nabla\Phi^R_\ga(\bx)=0}f(\bx).
  \]
  Let  $\fC_\ga:=\fC_\ga^{R}\rv_{R=1}$. The main goal  of this paper is to investigate the behavior of $\fC_\ga^R$ in the white noise limit, $R\nearrow \infty$. The  main result is the following.

\begin{theorem}\label{th: var_asy}  Fix an amplitude $\ga\in \eS(\bR)$.  Then the following hold.

\begin{enumerate}

\item There exists a universal  explicit constant $C_m(\ga)>0$ such that for any $f\in C^0_\cpt(\bR^m)$ we have
 \[
 \bE\lb \fC^R_\ga[f]\rb=C_m(\ga)R^m\int_{\bR^m} f(\bx) d\bx..
 \]
\item There exists  a constant $V_m(\ga)\geq 0$, that depends only on $m$ and $\ga$ such that for any $f\in C^0_\cpt(\bR^m)$ 
\[
\var\lb \fC^R_\ga[f]\rb\sim V_m(\ga)R^m\int_{\bR^m} f(\bx)^2 d\bx,\;\;\mbox{as $R\to\infty$}.
\]
\end{enumerate}\qed
\end{theorem}

The case $m=1$   of this  theorem  was proved by  M. Ancona  and T. Letendre \cite[Thm. 1.16]{AnLe20} and  L. Gass \cite[Thm.1.6]{Gass21}.  One immediate application of Theorem \ref{th: var_asy}  is a law of large numbers.

For any positive integer $N$ we denote by $\eL_N$ the random measure
\[
\eL_N:=\frac{1}{N^m}\fC_\ga^N.
\]
Theorem \ref{th: var_asy} shows  for any $f\in C^0_\cpt(\bR^m)$  we have
\[
\var\lb \eL_N[f]\rb \sim const \times N^{-m}\;\;\mbox{as $N\to\infty$}
\]
Since $\sum_{n\geq 1} n^{-m}<\infty$ for $m\geq 2$, we deduce from Borel-Cantelli the following  law of large numbers.

\begin{corollary}\label{cor: wLLN}   Let $m\geq 2$. Then for any $f\in C^0_\cpt(\bR^m)$, 
\begin{equation}\label{equidis}
\lim_{N\to\infty}\eL_N[f]=C_m(\ga)\blam\lb f\rb= C_m(\ga)\int_{\bR^m} f(\bx),\;\;\as\;\;\mbox{and}\;\;L^2,
\end{equation}
where $\blam$ denotes the Lebesgue measure.
\qed
\end{corollary}

As detailed in Appendix \ref{s: ergodic}  the above result can be rephrased as saying  that the random measures $\eL_N$ converge $\as$  and $L^2$ to  the deterministic measure $C_m(\ga)\blam$.    In particular, for any bounded  Borel  set $S$, the random variables $\eL_N[S]$ converge $\as$ to $C_m(\ga)\blam[S]$ a.s.. Thus, in the white noise limit $N\to\infty$, the critical points of $\Phi_\ga^N$ tend to equidistribute with high confidence. 	

For any bounded Borel set $S\subset \bR^m$ we have $\fC_\ga^N[S]=\fC_\ga[N\cdot S]$ and thus Corollary \ref{cor: wLLN} shows that 
\[
\lim_{N\to \infty}\frac{1}{N^m}\fC_\ga[N\cdot S]\to C_m(\ga)\blam[S] \;\; \as\;\;\mbox{and}\;\;L^2.
\]
In \cite{Nico_CLT} we  proved a Central Limit Theorem stating   that when $S=[0,\ell]^m$,  then the random variables
\[
N^{m/2}\Lp  \frac{1}{N^m}\fC_\ga[N\cdot S]-C_m(\ga)\blam[S]\Rp
\]
converge in distribution to a centered normal random variable with positive variance.

The new contribution  of Corollary \ref{cor: wLLN} the  explicit description of the a.s.  limit in (\ref{equidis}). The fact that this limit does indeed exist a.s., in  any dimension  $m$ follows   from the general theory of stationary random measures on $\bR^m$.   Thus Corollary \ref{cor: wLLN} also holds in the case $m=1$. Here are  the details.

For any $f\in C^0_\cpt(\bR^m)$ we have 
\[
\fC^N_\ga[f]=\fC_\ga[f_N],\;\;f_N(\bx):=f\lp N^{-1}\bx\rp.
\]
 Suppose  that $f\in C_\cpt^0(\bR^n)$ is nonnegative and
\[
\blam[f]=\int_{\bR^m} f(\bx) d\bx)>0.
\]   
Then the sequence 
\[
\vfi_N(\bx)=\frac{1}{N^m} f_N(\bx)
\]
is asymptotically stationary  in the precise sense defined in Appendix \ref{s: ergodic}.  The random measure $\fC_\ga$ is stationary and Theorem \ref{th: asy_sta} implies that 
\[
\lim_{N\to \infty}\frac{1}{\lambda\lb f\rb } \fC_\ga^N[f]
\]
 exists  a.s. and  $L^1$ and it is a random variable $\widehat{\fC}_\ga$ independent of $f$. Moreover, we can identify $\widehat{\fC}_{\ga}$ with a measurable function on  the space $\eM$ of locally finite Borel measures on $\bR^m$.  
 
 The distribution of the random measure  $\fC_\ga$ is a Borel probability measure  $\bP_{\fC_\ga}$ on $\eM$.  The additive group $\bR^m$ acts on $\eM$ by translations    and, since $\Phi_\ga$ is stationary, we deduce that $\bP_{\fC_\ga}$ is invariant under the above action of $\bR^m$. 
 
 In Appendix \ref{s: ergodic} we give an alternate ergodic description of $\widehat{\fC}_\ga$.  The fact that $\widehat{\fC}_\ga$  is constant  suggests that the action of $\bR^m$ on $\lp \eM,\bP_{\fC_\ga}\rp $ might be ergodic. There is some circumstantial evidence.
 
 The Gaussian function $\Phi_\ga$  defines a Gaussian measure $\bGamma$ on $C^2(\bR^m)$.  The  additive group $\bR^m$ acts   on  $C^2(\bR^m)$ by translations. Since the Gaussian function $\Phi_\ga$ is stationary,  we deduce that the above action is   $\bGamma$-preserving. Since the spectral measure of $\Phi_\ga$ is absolutely continuous with respect to the Lebesgue  measure, the above action of $\bR^m$ on $\lp C^2(\bR^m),\bGamma\rp$  is ergodic; see \cite{BE}. This is the fact used in 1960 by  V. Volkovski \cite{Vol}   to prove Corollary \ref{cor: wLLN} in  the case  $m=1$.   We refer to \cite[Sec. 11.5]{CraLe} for details.

The paper is organized as follows.  In Section \ref{s: 2}  we collect several probabilistic facts needed in the proof of Theorem \ref{th: var_asy}: the Kac-Rice formula, the regression formula  and  Lemma \ref {lemma: cont_gauss_int},  a H\"{o}lder continuity result concerning   certain functions  on the space of Gaussian measures on a fixed finite dimensional Euclidean space.  Section \ref{s: 4} contains  the proof of Theorem \ref{th: var_asy}.  We have subdivided this section into several parts corresponding to the conceptually distinct steps in the proof of Theorem \ref{th: var_asy}. Ultimately, two  facts lie behind this result: the stationarity of   the random function $\Phi_\ga$   and  the wery weak correlations  its values at far apart points.

\section{Some  probabilistic preliminaries}\label{s: 2} We collect in this section several  probabilistic facts   used throughout the paper.   We begin with the Kac-Rice formula, a key player in our proof.

\begin{theorem}\label{th: Kac-Rice} Let  $\eV\subset\bR^m$  an open set.  Suppose that  $F: \eV\to \bR$ a Gaussian random function that is $\as$ $C^2$ and such that  the Gaussian vector $\nabla F(\bv)$ is nondegenerate  for any $\bv\in \eV$. We denote by $p_{\nabla F(v)}$ is probability \emph{density}.  The following hold.
  
  \begin{enumerate}
  
  \item  The random function  $F$ is $\as$ Morse
  
  \item We set
  \begin{equation}\label{rand_crit_meas}
  \fC_F:=\sum_{\nabla F(\bv)=0}\delta_v
  \end{equation}
  Then $\fC_F$ is a random locally finite measure on $\mV$  in the sense of \cite{DaVere2} or \cite{Kalle_RM}.   For any  nonnegative measurable function $\vfi:\mV\to [0,\infty)$ we set
  \[
  \fC_F[\vfi]:=\int \vfi(v)\fC^F[dv]=\sum_{\nabla F(v)=0}\vfi(v).
  \]
  \item  For any box $B\subset \mV$,   the function $F$ $\as$ has no critical points on $\pa B$ and
  \begin{equation}\label{KR}
  \bE\lb \fC_F[\bsI_B\vfi]\rb=\int_B\bE\lb \vert\det \Hess F(\bv)\vert \, \rv\,\nabla F(\bv)=0\rb p_{\nabla F(v)}(0) \vfi(\bv)d\bv.
\end{equation}
We will refer to the  function 
\[
\rho_F(v):= \bE\lb \vert\det \Hess F(\bv)\vert \, \rv\,\nabla F(\bv)=0\rb p_{\nabla F(v)}(0)
\]
as the Kac-Rice density of $F$. \end{enumerate}
 \qed
  \end{theorem}

For proof and more details we refer to \cite{AT,AzWs}.   In applications, the conditional expectation  appearing in the Kac-Rice density is  computed using a classical trick.

Suppose that   $X,Y$ are jointly Gaussian centered random vectors valued in the finite dimensional Euclidean spaces $\bsX$ and $\bsY$. If  $X$ is nondegenerate, then  for any measurable  $f:\bsY\to\bR$ with polynomial growth at $\infty$  the conditional expectation $\bE\lb f(Y)\rv\, X=0\rb$ is computed  using the \emph{regression formula}
  \[
\bE\lb f(Y)\rv\, X=0\rb=\bE\lb f(Z)\rb
\]
where $Z$ is the  centered Gaussian vector $Y-\bE\lb Y\cond X\rb$\footnote{The conditional expectation $\bE\lb Y\cond X\rb$ is a random variable whereas $\bE\lb f(Y)\rv\,X=0\rb$ is a deterministic quantity, whence the difference in notation, ``$\cond$'' vs ``$\lv$''.}  whose variance is
\begin{equation}\label{reg2}
\var\lb Z\rb=\var\lb Y\rb-\cov\lb Y,X\rb\var\lb X\rb^{-1}\cov\lb X,Y\rb.
\end{equation}

 In the proof  we will need to compare expectations   with respect to different  Gaussian measures.  In  last part  of this section we describe a general  method  for  doing this.
 
   Suppose that $\bsV$ is an $m$-dimensional  real Euclidean space with inner product $(-,-)$. Denote by $S_1(\bsV)$  the unit sphere in $\bsV$ and  by $\Sym(\bsV)$ the space of symmetric operators $\bsV\to\bsV$, by $\Sym_{\geq 0}(\bsV)\subset \Sym(\bsV) $ the cone of nonnegative ones. For $A\in  \Sym_{\geq 0}(\bsV)$ we denote  by $\Gamma_A$ the centered  Gaussian measure  on $\bsV$  with variance $A$.
 
 The space $\Sym(\bsV)$ is equipped with an inner product
 \[
 \lp A, B\rp_\op=\tr(AB),\;\;\forall A, B\in \Sym(\bsV).
 \]
 We denote by $\Vert-\Vert_\op$ the associated norm.

 We have a natural map $\Sym_{\geq 0}(\bsV)\to \Sym_{\geq 0}(\bsV)$, $A\mapsto  A^{1/2}$.  We will need the following result,  \cite[Prop.2.1]{HeAn}. 
 
 \begin{proposition} For any $\mu>0$ and $\forall A,B\in \Sym_{\geq 0}(\bsV)$,  such that $A^{1/2}+B^{1/2}\geq \mu\one$ we have
  \begin{equation}\label{sq_holder}
\mu \lV A^{1/2}-B^{1/2}\rV_\op \leq \lV A-B\rV^{1/2}_\op. 
 \end{equation}\qed
\end{proposition}
 
 \begin{lemma}\label{lemma: cont_gauss_int} Fix $A_0\in \Sym_{\geq 0}(\bsV)$ such that $A_0^{1/2}\geq \mu_0\one$, $\mu_0>0$. Suppose that $f:\bV\to\bR$ is a  locally Lipschitz function that is  homogeneous of degree $k\geq  1$.  For $A\in \Sym_{\geq 0}(\bsV)$ we set
 \[
 \eI_A(f):=\int_\bsV f(\bv)\bf \Gamma_A\lb d\bv\rb.
 \]
  Then for and $R\geq \Vert A_0\Vert_\op$   there exists a constant $C=C(f,R,\mu_0)>0$  with the following  property:  for any $A\in \Sym_{\geq 0}(\bsV)$ such that $\Vert ,A\Vert_\op\leq R$
 \begin{equation}\label{cont_gauss_int}
 \lv \eI_{A_0}(f)-\eI_A(f)\rv \leq C\Vert A-A_0\Vert^{1/2}\leq C(k,R) \Vert A-A_0\Vert_\op^{1/2}.
 \end{equation}
 In other words, $A\mapsto \eI_A(f)$ is locally H\"{o}lder continuous with exponent  $1/2$ in the  open set $\Sym_{>0}\lp \bsV\rp$.
 \end{lemma}
 
 \begin{proof}The function $f$ is Lipschitz on the ball 
 \[
 B_R(\bsV):=\big\{\, \bv \in \bsV; \Vert \bv\Vert\leq R\,\big\},
 \]
  so there exists $L=L(R)>0$ such that
 \begin{equation}\label{Lipschitz_gauss}
 \lb f(\bu)-f(\bv)\rv \leq L\Vert\bu-\bv\Vert,\;\;\forall \bu,\bv\in B_R(\bsV).
 \end{equation}
 Note that
 \[
 \eI_A(f)=\int_\bsV f\lp A^{1/2}\bv\rp \Gamma_{\one}\lb d\bv\rb,
 \]
 so
 \[
 \lv \eI_{A_0}(f)-\eI_A(f)\rv\leq  \int_\bsV \lv f\lp A^{1/2}\bv\rp -f\lp A_0^{1/2}\bv\rp\rv\;\Gamma_{\one}\lb d\bv\rb
 \]
\[
=\underbrace{\frac{1}{(2\pi)^{m/2}}\left(\int_0^\infty r^{n+k-1} e^{-r^2/2} dr\right)}_{C_{m,k}}\int_{S_1(\bsV)} \lv f\lp A^{1/2}\bv\rp -f\lp A_0^{1/2}\bv\rp\rv\vol_{S_1(\bsV)}\lb d\bv\rb
\]
\[
\stackrel{(\ref{Lipschitz_gauss})}{\leq} C_{m,k}L(R) \int_{S_1(\bsV)} \Vert A^{1/2}-A_0^{1/2}\Vert_\op\vol_{S_1(\bsV)}\lb d\bv\rb\stackrel{(\ref{sq_holder})}{\leq} C(k,R,\mu_0) \Vert A-A_0\Vert_\op^{1/2}.
\]
\end{proof}

 \section{Proof of Theorem \ref{th: var_asy}}\label{s: 4}
 
 \subsection{Proof of Theorem \ref{th: var_asy}(i)}  To compute the expectation of $\fC^R_\ga[f]$  we rely on the Kac-Rice formula.  Note that $\bx$  is a critical point of $\Phi_\ga^R$ iff $R^{-1}\bx$ is a critical point of $\Phi_\ga$ so that, for any $f\in C^0_\cpt(\bR^m)$, we have
  \[
  \fC_\ga^R[f]=\fC_\ga[ f_R],\;\;f_R(\bx)= f(\bx/R).
  \]
  We want to apply  the Kac-Rice formula to $\Phi_\ga$.
  
  For $k\in \bN$ we denote by $D^k\Phi_\ga$  the $k$-th order differential of $\Phi_\ga$

\begin{proposition}\label{prop: ample}   Let $N\in \bN$. Then the following hold.

\begin{enumerate}

\item The function $\Phi_\ga$ is $N$-ample, i.e., for any  distinct  points $\bx_1,\dotsc,\bx_N\in \bR^m$, the Gaussian vector
\[
 \lp \Phi_\ga(\bx_1),\dotsc,\Phi_\ga(\bx_N)\rp.
 \]
 is nondegenerate.
 \item  The function $\Phi_\ga$ is $J_N$-ample\footnote{$N$-th jet ample},  i.e.,  the Gaussian vector
 \[
\Phi_\ga(\bx)\oplus D\Phi_\ga(\bx)\oplus \cdots \oplus D^N\Phi_\ga(\bx)
\]
 is nondegenerate for any $\bx\in \bR^m$.
 \end{enumerate}
 \end{proposition}
 
 \begin{proof} (i) Since $w_{\ga,m}\lp\lv\xi\rv\rp=\ga\lp |\xi|\rp^2$ is positive on  an open neighborhood of $0\in \bR^m$ we deduce  from  \cite[Thm. 6.8]{Wen} that if $\bx_1, \dotsc,\bx_N\in \bR^m$    are distinct  points,   then   the symmetric  $N\times N$ matrix 
 \[
 \lp K_\ga(\bx_i-\bx_j)\rp_{1\leq i,j\leq N}
 \]
 is \emph{positive} definite.  This matrix  is the variance matrix of the Gaussian vector 
 \[
 \lp \Phi_\ga(\bx_1),\dotsc,\Phi_\ga(\bx_N)\rp.
 \]
  (ii)   Observe that for any multi-indices  $\alpha\in\lp\bZ_{\geq 0}\rp^m$,  we have
   \[
   \bE\lp \pa^\alpha\Phi_\ga(\bx)\pa^\beta\Phi_\ga (\bx)\rp=\pa^\alpha_x\pa^\beta_y\bsK_\ga(\bx-\by\rp\big\vert_{\bx=\by}
   \]
   \[
   =\int_{\bR^m} \xi^\alpha\xi^\beta\mu_\ga\lb d\xi\rb,\;\;\xi^\alpha:=\xi_1^{\alpha_1}\cdots \xi_m^{\alpha_m}
   \]
This shows that for  any  $N\in \bN$ and any $\bx\in \bR^n$   the variance the  Gaussian vector $\lp \pa^\alpha\Phi_\ga(\bx)\rp_{|\alpha|\leq N}$ is  the Gramian matrix  of the functions $\lp \xi^\alpha\rp_{|\alpha|\leq N}$ with respect to the inner product in $L^2\lp \bR^m,\mu_\ga\rp$. Since $\ga(0)=1$ we deduce that the functions $\xi^\alpha$ are linearly independent in  $L^2\lp \bR^m,\mu_\ga\rp$ so  the determinant of their Gramian matrix is nonzero. Hence the Gaussian vector 
\[
\Phi_\ga(\bx)\oplus D\Phi_\ga(\bx)\oplus \cdots \oplus D^N\Phi_\ga(\bx)
\]
is nondegenerate, for any $k\in\bN$ and any $\bx\in\bR^m$.
  \end{proof}
  
   Results of Ancona-Letendre \cite{AnLe23} or Gass-Stecconi \cite{GS23}   show that the $J_N$-ampleness of $\Phi_\ga$ implies that for  any function $f\in C^0_\cpt(\bR^m)$, the random variable  $\fC_\ga[ f]$  has finite moments of any order.  
  
  Proposition \ref{prop: ample}  shows that the Gaussian vector  $\nabla \Phi_\ga(\bx)$ is nondegenerate  for any $\bx\in \bR^m$. For any multi-indices $\alpha,\beta\in \lp \bZ_{\geq 0}\rp^m$ we have
\begin{equation}\label{cov_ga}
\begin{split}
\bE\lb \pa^\alpha\Phi_\ga(\bx)\pa^\beta\Phi_\ga(\by)\rb_{\bx=\by} =\pa^\alpha_\bx\pa^\beta_\by\eK^\ga(\bx,\by)\lv_{\bx=\by}\\
=\frac{(-1)^{|\beta||}\ii^{|\alpha|+|\beta|}}{(2\pi)^m}\int_{\bR^m}\xi^{\alpha+\beta}\ga(|\xi|)^2d\xi,\;\;\xi^\alpha=\xi_1^{\alpha_1}\cdots \xi_m^{\alpha_m}.
\end{split}
\end{equation}
For any multi-index  $\alpha\in\lp \bZ_{\geq 0}\rp^m$  we set
\[
M^\ga_{\alpha}:=\int_{\bR^m}\xi^\alpha \mu_\ga\lb d\xi\rb=\frac{1}{(2\pi)^m}\int_{\bR^m}\xi^{\alpha}\ga(|\xi|)^2 d\xi.
\]
We say that the multi-index $\alpha=(\alpha_1,\dotsc, \alpha_m)$ is \emph{even} \index{multi-index! even} \index{multi-index! odd} if $\alpha_j$ is even for any $j=1,\dotsc, m$.   The multi-index $\alpha$ is called \emph{odd} if it is not even. The radial symmetry of $\ga\lp |\xi|\rp$ implies that 
\begin{equation}\label{M_gamma_1}
M^\ga_{\alpha}=0\;\;\mbox{if $\alpha$ is odd}.
\end{equation}
Using spherical coordinates on $\bR^m$ we deduce that for any $\alpha$ we have
\begin{equation}\label{mom_spec_meas}
M^\ga_{\alpha}=\frac{1}{(2\pi)^m}\left(\int_0^\infty r^{m-1+|\alpha|}  \ga(r)^2 dr\right)\times \underbrace{\int_{S^{m-1}}\xi^\alpha\vol_{S^{m-1}}\lb d\xi\rb}_{=: \bm_\alpha},
\end{equation}
$S^{m-1}=S_1(\bR^m)$. Note that  $\bm_\alpha$ is independent of $\ga$.    If we let $\ga_0:=(2\pi)^{m/2} e^{-\frac{t^2}{4}}$, then
\[
M^{\ga_0}_\alpha=\int_{\bR^m} \xi^\gamma e^{-|\xi|^2/2} d\xi=(2\pi)^{m/2}\prod_{j=1}^m \int_\bR \xi^{\alpha_j}\bgamma_1\lb d\xi\rb
\]
where $\bgamma_1$ denotes the Gaussian  measure on $\bR$ with mean zero and variance $1$. If $\alpha$ is even, $\alpha=2\kappa$, then
\[
M^{\ga_0}_{2\kappa}=(2\pi)^{m/2}\prod_{j=1}^m (2\kappa_j-1)!!.
\]
On the other hand, using (\ref{mom_spec_meas}) we deduce
\[
M^{\ga_0}_{2\kappa}=\bm_{2\kappa}\int_0^\infty r^{m+2|\kappa|-1} e^{-r^2/2}  dr=\sqrt{\frac{\pi}{2}}\bm_{2\kappa}\int_\bR |x|^{m+2|\kappa|-1}\bgamma_1\lb dx\rb
\]
\[
=2^{|\kappa|+m/2-1}\bm_{2\kappa} \Gamma\lp |\kappa|+m/2\rp.
\]
Hence 
\begin{equation}\label{M_gamma_2}
\bm_{2\kappa}=\frac{(2\pi)^{\frac{m}{2}}\prod_{j=1}^m (2\kappa_j-1)!!}{2^{|\kappa|+m/2-1}\Gamma\lp |\kappa|+m/2\rp}=\frac{2\prod_{j=1}^m\Gamma(\kappa_j+1/2)}{\Gamma\lp |\kappa|+m/2\rp}.
\end{equation}
Above  we used the classical identities
\[
\Gamma(1/2)=\pi^{1/2},\;\;\Gamma\lp x+1\rp=x\Gamma(x).
\]
For every $k\in \bZ_{\geq 0}$ we set
\[
I_k(\ga):= \int_0^\infty r^k\ga(r)^2 dr.
\]
We deduce 
\begin{equation}\label{M_gamma_3}
(2\pi)^m M^\ga_{2\kappa}= I_{m-1+2|\kappa|}(\ga) \frac{2\prod_{j=1}^m\Gamma(\kappa_j+1/2)}{\Gamma\lp |\kappa|+m/2\rp}.
\end{equation} 
We set
\begin{equation}\label{basic_inv}
s_m:=\int_{\bR^m}\mu_\ga\lb d\xi\rb,\;\;d_m=\int_{\bR^m}\xi_1^2\mu_\ga\lb d\xi\rb,\;\;h_m:=\int_{\bR^m}\xi_1^2\xi_2^2\mu_\ga\lb d\xi\rb.
\end{equation}
Then
\begin{equation}\label{spec_mom_0}
\int_{\bR^m}\ga(|\xi|)^2 d\xi=\frac{2\pi^{m/2} }{\Gamma(m/2)}I_{m-1}(\ga)=(2\pi)^m s_m,
\end{equation}
\begin{equation}\label{spec_mom_1}
\int_{\bR^m}\xi_j^2\ga(|\xi|)^2 d\xi= \frac{2\pi^{m/2} }{\Gamma(m/2+1)}I_{m+1}(\ga)=(2\pi)^m d_m,\;\;\forall j,
\end{equation}
\begin{equation}\label{spec_mom_2}
\int_{\bR^m}\xi_j^2\xi_k^2\ga(|\xi|)^2 d\xi =\frac{(2\pi)^{m/2} }{\Gamma(m/2+2)}I_{m+3}(\ga)=(2\pi)^m h_m,\;\;\forall j\neq k,
\end{equation}
\begin{equation}\label{spec_mom_3}
\int_{\bR^m}\xi_j^4\ga(|\xi|)^2 d\xi= \frac{6 \pi^{m/2} }{\Gamma(m/2+1)}I_{m+3}(\ga)=3(2\pi)^m h_m,\;\;\forall j.
\end{equation}
Using (\ref{cov_ga}) and (\ref{M_gamma_1}) we deduce that for any $\bx\in \bR^m$ \emph{the Gaussian vectors $\nabla\Phi_\ga(\bx)$ and $\Hess_{\Phi_\ga}(\bx)$ are independent}.  Hence
\[
\bE\lb \det \Hess_{\Phi_\ga}(\bx)|\,\lv \nabla\Phi_\ga(\bx)=0\rb= \bE\lb \det \Hess_{\Phi_\ga}(\bx)|\rb.
\]
Using (\ref{cov_ga})and (\ref{spec_mom_1}) we deduce that the variance matrix of $\nabla \Phi_\ga(\bx)$ is
\begin{equation}\label{grad_iso}
\var\lb \nabla\Phi_\ga(\bx)\rb=d_m\one_m,\;\;\forall \bx\in\bR^m,
\end{equation}
where $\one_m$ denotes the identity $m\times m$ matrix. Hence
\[
p_{\nabla\Phi_\ga(\bx)}(0)=(2\pi d_m)^{-m/2}.
\]
The space $\Sym(\bR^m)$ of real symmetric $m\times m$ matrices is equipped with an inner product $(A,B)=\tr(A)$. Moreover, the linear functions $\ell_{ij},\omega_{ij}: \Sym(\bR^m)\to\bR$, $1\leq i\leq j\leq m$,
\begin{equation}\label{xiija}
\ell_{ij}(A)=a_{ij},\;\;\omega_{ij}(A)=\begin{cases}
a_{ii}, &i=j,\\
\sqrt{2}a_{ij}, & i<j
\end{cases}
\end{equation}
define coordinates on $\Sym(\bR^m)$ that are \emph{orthonormal} with respect to the above inner product.  We set
\begin{equation}\label{xiij}
L_{ij}(\bx):=\ell_{ij}\lp \Hess_\Phi(\bx)\rp,\;\;\Omega_{ij}(\bx):=\omega_{ij}\lp \Hess_\Phi(\bx)\rp.
\end{equation}
Then
\[
\bE\lb \pa^2_{x_i x_j}\Phi_\ga(x)\pa^2_{x_kx_\ell}\Phi_\ga(\bx)\rb=\frac{1}{(2\pi)^m}\int_{\bR^m} \xi_i\xi_j\xi_k\xi_\ell a\lp|\xi|^2\rp d\xi ,\;\;i\leq j,\;\;k\leq \ell.
\]
Note that if  $i<j$,    then  the above integral is nonzero iff $(i,j)=(k,\ell)$ in which case
\[
\bE\lb L_{ij}(\bx)L_{ij}(\bx)\rb=\bE\lb \lp\pa^2_{x_i x_j}\Phi_\ga(x)\rp^2\rb
\]
\[
= \frac{1}{(2\pi)^m}\int_{\bR^m} \xi_i^2\xi_j^2\ga\lp |\xi|^2\rp d\xi\stackrel{(\ref{spec_mom_2})}{=}h_m.
\]
If $i=j$, then the above integral is nonzero iff $k=\ell$,  in which case we deduce from (\ref{spec_mom_2}) and (\ref{spec_mom_3})
\[
\bE\lb \pa^2_{x_i }\Phi_\ga(x)\pa^2_{x_k}\Phi_\ga(\bx)\rb=\begin{cases}
h_m & i\neq k,\\
3h_m, & i=k.
\end{cases}
\]
The above equalities can be rewritten in the more compact form
\begin{equation}\label{hess_iso}
\bE\lb L_{ij}(\bx)L_{k\ell}(\bx)\rb= h_m\lp\delta_{ij}\delta_{k\ell}+\delta_{ik}\delta_{j\ell}+ \delta_{i\ell}\delta_{jk}\rp,\;\;\forall i\leq j,\;k\leq \ell.
\end{equation}
These equalities show that the off-diagonal entries  of $\Hess_\Phi$  are $\iid$, and also independent of the diagonal entries.  The diagonal  entries  have identical distributions  but  are dependent.  The parameter $h_m$ describes the various variances and covariances. 

The Gaussian measure  on $\Sym(\bR^m)$  determined by these  covariance equalities is  invariant with respect to the action of $O(m)$  by conjugation  on $\Sym(\bR^m)$. 

For $v>0$  denote  by $\eS_m^{v}$  the space $\Sym(\bR^m)$ equipped with the $O(m)$-invariant Gaussian  measure on $\Sym(\bR^m)$ determined by the covariances
\[
\bE\lb L_{ij}(\bx)L_{k\ell}(\bx)\rb= v\lp\delta_{ij}\delta_{k\ell}+\delta_{ik}\delta_{j\ell}+ \delta_{i\ell}\delta_{jk}\rp,\;\;\forall i\leq j,\;k\leq \ell.
\]
Hence
\[
 \bE\lb \det \Hess_{\Phi_\ga}(\bx)|\rb=  \bE_{\eS_m^{h_m}}\lb |\det H|\rb.
 \]
  We deduce from the Kac-Rice formula (\ref{KR}) that 
\[
\bE\lb \fC_\ga[f_R] \rb=\int_{\bR^m} \bE_{\eS_m^{h_m}}\lb |\det H|\rb p_{\nabla \Phi(\bx)} (0) f_R(\bx)  \blam\lb d\bx\rb
\]
\[
\stackrel{(\ref{grad_iso})}{=}(2\pi d_m)^{-m/2} \bE_{\eS_m^{h_m}}\lb |\det H|\rb\int_{\bR^m} f_R(\bx) dx.
\]
Note that
\[
\int_{\bR^m} f_R(\bx) dx. =R^m\int_{\bR^m} f(\by) d\by.
\]
Using the linear change in variables $X=(2h_m)^{-1/2}H$ we deduce
\[
(2\pi d_m)^{-m/2} \bE_{\eS_m^{h_m}}\lb |\det H|\rb=\underbrace{\left(\frac{h_m}{\pi d_m}\right)^{m/2} \bE_{\eS_m^{1/2}}\lb |\det X|\rb}_{=:C_m(\ga)} \, .
\]
Hence 
\begin{equation}\label{crit_isotrop}
\bE\lb \fC_\ga^R[f] \rb= C_m(\ga)R^m\int_{\bR^m} f(\by) d\by..
\end{equation}

\begin{remark} One can prove that  as $m\to\infty$
\[
C_m(\ga)\sim  2^{\frac{5}{2}}\Gamma\left(\frac{m+3}{2}\right)\left(\frac{h_m}{\pi d_m}\right)^{m/2}\left(\frac{1}{m+1}\right)^{1/2}.
\]
Using (\ref{spec_mom_1}) and (\ref{spec_mom_2}) we deduce
\[
\frac{h_m}{d_m}=\frac{\Gamma(1+m/2)}{\Gamma(2+m/2)}\times \frac{I_{m+3}(\ga)}{I_{m+1}(\ga)}=\frac{2I_{m+3}(\ga)}{(m+2)I_{m+1}(\ga)}.
\]
Hence 
\begin{equation}\label{Cmga1}
\begin{split}
C_m(\ga)\sim 2^{5/2}\left(\frac{h_m(\ga)}{d_m(\ga)}\right)^{m/2}  \Gamma\left(\frac{m+3}{2}\right)m^{-1/2}\hspace{1in} \\
\sim2^{\frac{m+5}{2}}\left(\frac{I_{m+3}(\ga)}{(m+2)I_{m+1}(\ga)}\right)^{m/2} \Gamma\left(\frac{m+3}{2}\right)m^{-1/2}\;\;\mbox{as $m\to\infty$}.
\end{split}
\end{equation}
The constant $C_m(\ga)$ tends to grow very fast\footnote{Think super factorial.} as $m\to\infty$,  but its large $m$  behavior  depends on the tail  of the  amplitude $\ga$. Roughly speaking, the slower the decay at $\infty$  of $\ga$  the faster the growth of $C_m(\ga)$.  \qed
\end{remark}  

\subsection{An integral formula for the variance}  
We need to introduce some notation. Set

 \begin{itemize}

 \item Define
\[
\;\hPhi:\bR^m\times \bR^m\to\bR,\;\;\hPhi(\bx,\by)=\Phi_\ga(\bx)+\Phi_\ga(\by),\;\;\hat{\fC}=\fC_{\hPhi},
\]
\[
 \hh(\bx,\by):=\Hess_{\hPhi}(\bx,\by),\;\;H(\bx):=\Hess_{\Phi_\ga}(\bx).
\]
\item Choose  an independent copy $\Psi_\ga$ of $\Phi_\ga$ and  set
 \[
 \tphi(\bx,\by):=\Phi_\ga(\bx)+\Psi_\ga(\by),\;\;\tH(\bx,\by):=\Hess_{\tphi}(\bx,\by),\;\;\tilde{\fC}=\fC_{\tphi}.
\]
 
 \item Set $\Vert f\Vert:=\Vert f\Vert_{C^0(\bR^m)}$. 
 
 \item Set
 \[
 \eX= \bR^m\times \bR^m\setminus \Delta=\big\{\, (\bx,\by)\in \bR^m\times\bR^m;\;\;\bx\neq \by\,\big\}.
 \]
 \end{itemize}
 Observe that  the random function on  $ \hPhi(\bx,\by)$ is \emph{stationary} with respect to the action of $\bR^{m}$ on $\bR^m\times \bR^m$ itself by translations
 \begin{equation}\label{T_v}
 T_{\bv}(\bx,\by)=\lp \bx+\bv,\by+\bv\rp,\;\;\bx,\by,\bv\in \bR^m,
 \end{equation}
 where as $\tphi$ is stationary with respect to the action by translations of $\bR^{2m}$ on itself.
 
 We have 
 \[
 \hat{\fC}[\bsI_{\eX} f_R^{\boxtimes 2}]=\sum_{\substack{\nabla\Phi_\ga(\bx)=\nabla\Phi_\ga(\by)=0,\\ \bx\neq \by}} f_R(\bx)f_R(\by)=\fC_\ga[f_R]^2-\fC_\ga[f_R^2].
 \]
 Bulinskaya's lemma implies that  
 \[
 \bP\lb \exists \bx:\;\;\nabla\Phi_\ga(\bx)=\nabla\Psi_\ga(\bx)=0\rb=0
 \]
 and we deduce
  \[
 \tilde{\fC}[\bsI_{\eX} f_R^{\boxtimes 2}]=\sum_{\substack{\nabla\Phi_\ga(\bx)=\nabla\Psi_\ga(\by)=0,\\ \bx\neq \by}} f_R(\bx)f_R(\by)
 \]
 \[
 = \sum_{\nabla\Phi_\ga(\bx)=\nabla\Psi_\ga(\by)=0} f_R(\bx)f_R(\by)= \fC[f,\Phi_\ga]\rb \fC[f,\Psi_\ga],\;\;\as.
 \]
 Hence
 \[
 \bE\lb \fC[f_R,\Phi_\ga] \cdot \fC[f,\Psi_\ga]\rb=\bE\lb   \fC[f_R,\Phi_\ga]\rb\,\cdot\, \bE\lb   \fC[f_R,\Psi_\ga]\rb=\bE\lb   \fC[f_R,\Phi_\ga]\rb^2
 \]
 so that
 \begin{equation}\label{var_quad}
\bE\lb  \hat{\fC}^R[\bsI_{\eX} f_R^{\boxtimes 2}]\rb -\bE\lb  \tilde{\fC}^R[\bsI_{\eX} f_R^{\boxtimes 2}]\rb= \underbrace{\bE\lb \fC_\ga[f_R]^2\rb-\bE\lb \fC_\ga[f_R]\rb^2}_{=\var\lb  \lb \fC_\ga [f_R]\rb}-\bE\lb\fC_\ga[f_R^2]\rb.
\end{equation}
We have seen that 
\[
\bE\lb \fC_\ga[f_R^2]\rb=R^mC_m(\ga)\int_{\bR^m}f^2(\bx) d\bx
\]
so we have to show that
\begin{equation}\label{I(R)}
I(R):=\bE\lb  \hat{\fC}[\bsI_{\eX} f_R^{\boxtimes 2}]\rb -\bE\lb  \tilde{\fC}[\bsI_{\eX} f_R^{\boxtimes 2}]\rb\sim Z_m(\ga)R^{m}\int_{\bR^m} f(\bx)^2 d\bx\;\;\mbox{as $R\to\infty$}
\end{equation}
for some constant  $Z_m(\ga)\in \bR$ that depends only on $m$ and $\ga$.

\begin{lemma}\label{lemma: 2_grad} For any $\bx,\by\in\bR^m$, $\bx\neq \by$, the Gaussian vector $\nabla\hphi(\bx,\by)$ is nondegenerate. 
\end{lemma}

\begin{proof} We have
\[
\var\lb  \nabla\hphi(\bx,\by)\rb =\left[ \begin{array}{cc}
\var\lb \nabla\Phi_\ga(\bx)\rb & \cov\lb \nabla\Phi_\ga(\bx),\nabla \Phi_\ga(\by)\rb\\
&\\
\cov\lb \nabla\Phi_\ga(\by),\nabla \Phi_\ga(\bx)\rb & \var\lb \nabla\Phi_\ga(\by)\rb
\end{array}
\right].
\]
As shown in (\ref{grad_iso}), for any $\bx\in \bR^n$ we have
\[
\var\lb \nabla\Phi_\ga(\bx)\rb=d_m\one_m,\;\;d_m=\int_{\bR^n}\xi_1^2 \mu_\ga\lb d\xi\rb.
\]
We have
\[
\cov\lb \nabla\Phi_\ga(\bx),\nabla \Phi_\ga(\by)\rb=\lp \pa_{x_j}\pa_{y_k}\bsK_\ga (\bx-\by)\rp_{1\leq j,k\leq m}
\]
and 
\begin{equation}\label{cov_grad_hphi}
\  \pa_{x_j}\pa_{y_k}\bsK_\ga (\bx-\by)=\int_{\bR^m}e^{-\ii \lan \xi,\bx-\by\ran }\xi_j\xi_k \mu_\ga\lb  d\xi\rb.
 \end{equation}
 Since $\Phi_\ga$ is stationary it suffice to consider only the case $\bx=0$. On the other hand, $\Phi_\ga$ is   $O(m)$-invariant so, up to a rotation, we can assume that $\bx-\by=-t\be_1$, $t\neq 0$, where $\{\be_1,\dotsc ,\be_m\}$ is the canonical basis of $\bR^m$. Hence
 \[
  \pa_{x_j}\pa_{y_k}\bsK_\ga (\bx-\by) =\int_{\bR^m} e^{\ii t\xi_1} \xi_j\xi_k \mu_\ga \lb  d\xi\rb.
 \]
Let us observe that if $j\neq k$, then  either $j\neq 1$, or $k\neq 1$.  Suppose $j\neq 1$. The function  $e^{\ii t\xi_1} \xi_j\xi_k$ is odd with respect to the reflection $\xi_j\mapsto -\xi_j$ so
\[
\pa_{x_j}\pa_{y_k}\bsK_\ga (\bx,\by) =\int_{\bR^m} e^{\ii t\xi_1} \xi_j\xi_k \mu_\ga \lb  d\xi\rb=0,\;\;\forall j\neq k.
\]
If $j=k$, then
\[
d_m(j):=\pa_{x_j}\pa_{y_j}\bsK_\ga (\bx,\by)= \int_{\bR^m} e^{\ii t\xi_1} \xi_j^2\mu_\ga \lb  d\xi\rb=  \int_{\bR^m} \cos (t\xi_1) \xi_j^2\mu_\ga \lb  d\xi\rb
\]
and we deduce\footnote{At this point we use the fact that $\ga\lp |\xi|\rp >0$ for $|\xi|$ sufficiently small.}
\[
|v_m(j)|\leq  \int_{\bR^m} \lv \cos (t\xi_1)\rv \xi_j^2\mu_\ga \lb  d\xi\rb<  \int_{\bR^m}  \xi_j^2\mu_\ga \lb  d\xi\rb=d_m.
\]
After a reordering  
\[
\begin{array}{c}
\lp \pa_{x_1}\Phi_\ga(\bx),\dotsc,\pa_{x_m}\Phi_\ga(\bx),\pa_{y_1}\Phi_\ga(\by),\dotsc \pa_{y_m}\Phi_\ga(\by)\rp\\
\downarrow\\
\lp \pa_{x_1}\Phi_\ga(\bx),\pa_{y_1}\Phi_\ga(\by),\dotsc,\pa_{x_m}\Phi_\ga(\bx),\pa_{y_m}\Phi_\ga(\by)\rp
\end{array}
\]
we see that 
\[
\var\lb \nabla\hphi (\bx,\by)\rb =\bigoplus_{j=1}^m\underbrace{\left[
\begin{array}{cc}
d_m & d_m(j)\\
d_m(j) & d_m
\end{array}
\right]}_{=:V_j}.
\]
Note that, for each $j$, the symmetric matrix $V_j$ is positive definite since  
\[
\det V_j=d_m^2-d_m(j)^2>0.
\]

\end{proof}
Since $\nabla \hPhi(\bx,\by)$ is nondegenerate   for $\bx\neq \by$ we can use the Kac-Rice formula to compute  $\bE\lb  \hat{\fC}^R[\bsI_{\eX} f_R^{\boxtimes 2}]\rb$.  We  deduce  that  for any $R>0$
 \begin{equation}\label{KR_cov_1aa}
\begin{split}
\bE\lb  \hat{\fC}[\bsI_{\eX} f_R^{\boxtimes 2}]\rb\hspace{5cm}\\
=  \int_{\bR^m\times \bR^m\setminus\Delta } \underbrace{\bE\lb \vert\det \hh(\bx,\by)\vert \rv \nabla \hphi(\bx,\by)=0\rb p_{\nabla\hphi(\bx,\by)}(0)}_{=\hrho(\bx,\by)} f_R^{\boxtimes 2}(\bx,\by)\blam\lb d\bx d\by\rb.
\end{split}
\end{equation}
Since $\hPhi$ is invariant under the translations (\ref{T_v}) we deduce that $\hrho$ depends only on $\bx-\by$. 

The gradient $\nabla\tphi(\bx,\by)$ is nondegenerate for any $\bx,\by$ and invoking Kac-Rice again we obtain
\begin{equation}\label{KR_cov_1aat}
\begin{split}
\bE\lb  \tilde{\fC}[\bsI_{\eX} f_R^{\boxtimes 2}]\rb\hspace{5cm}\\
=  \int_{\bR^m\times \bR^m\setminus\Delta } \underbrace{\bE\lb \vert\det \tH(\bx,\by)\vert \rv \nabla \tphi(\bx,\by)=0\rb p_{\nabla\tphi(\bx,\by)}(0)}_{=\trho(\bx,\by)} f_R^{\boxtimes 2}(\bx,\by)\blam\lb d\bx d\by\rb.
\end{split}
\end{equation}
The function  $\trho_R(\bx,\by)$ is independent of $\bx,\by$ since the random function $\tphi$  is stationary. Thus
\begin{equation}\label{KR_cov_2_per}
\begin{split}
I(R)=\int_{\eX}\lp \underbrace{\hrho(\bx,\by)-\trho(\bx,\by)\rp}_{=:\Delta(\bx,\by)} f_R(\bx)f_R(\by)\blam\lb d\bx d\by\rb.
\end{split}
 \end{equation}

There is a serious issue  concerning $\hrho(\bx,\by)$ namely  it blows up as $(\bx,\by)$ approaches the diagonal so this Kac-Rice density may not be locally integrable.

\subsection{Off-diagonal  behavior}   We  first describe the behavior of $\Delta(\bx,\by)$ away from the diagonal. Note that $\Delta$ depends only on the $\bx-\by$.

 For every $\bz\in \bR^m$ we set 
 \[
 T(z):=\sum_{|\alpha|\leq 4}\lv \pa^\alpha \bsK_\ga(\bz)\rv.
 \]
 Since $\bsK_\ga$ is a Schwartz  function we deduce  that
 \[
 T(\bz)= O\lp |\bz|^{-\infty}\rp\;\;\mbox{as $|\bz| \to \infty$}.
 \]
 This means that
 \[
 \forall p>0,\;\; T(\bz)= O\lp |\bz|^{-p}\rp\;\;\mbox{as $|\bz| \to \infty$}.
 \]
 Observe that
 \[
 \var\lb \nabla \tphi(\bx,\by)\rb=\left[
 \begin{array}{cc}
 \var\lb \nabla\Phi_\ga(\bx)\rb  & 0\\
 0 & \var\lb \nabla\Phi_\ga(\by)\rb
 \end{array}
 \right]=d_m\one_{2m},
 \]
 and
 \[
  \var\lb \nabla \hphi(\bx,\by)\rb=\left[
 \begin{array}{cc}
 \var\lb \nabla\Phi_\ga(\bx)\rb  & \cov\lb \nabla\Phi_\ga(\bx),\nabla\Phi_\ga(\by)\rb\\
 & \\
 \cov\lb \nabla\Phi_\ga(\by),\nabla\Phi_\ga(\bx)\rb & \var\lb \nabla\Phi_\ga(\by)\rb
 \end{array}
 \right]
 \]
 \[
 = \var\lb \nabla \tphi(\bx,\by)\rb+ \underbrace{\left[
 \begin{array}{cc}
0& \cov\lb \nabla\Phi_\ga(\bx),\nabla\Phi_\ga(\by)\rb\\
 & \\
 \cov\lb \nabla\Phi_\ga(\by),\nabla\Phi_\ga(\bx)\rb &0 
 \end{array}
 \right]}_{=: R_\nabla(\bx,\by)}.
 \]
Hence
 \begin{equation}\label{KR_cov_3a_per}
 \lV  \var\lb \nabla \hphi(\bx,\by)\rb-  \var\lb \nabla \tphi(\bx,\by)\rb\rV_\op=\Vert R_\nabla(\bx,\by)\Vert_\op=O\lp T_R(\bx-\by)\rp,
\end{equation}  
The operators $\var\lb \nabla \hphi(\bx,\by)\rb$  and $\var\lb \nabla\tphi(\bx,\by)\rb$  are  invertible  for $\bx\neq \by$. In particular
 \begin{equation}\label{invargrad}
\begin{split}
\var\lb   \nabla\hphi(\bx,\by)\rb^{-1}=\Lp  \var\lb   \nabla\tphi(\bx,\by)\rb+R_\nabla(\bx,\by)\Rp^{-1}\\
=\Lp \one + \var\lb   \nabla\tphi(\bx,\by)\rb^{-1}R_\nabla(\bx,\by)\Rp^{-1}\var\lb   \nabla\tphi(\bx,\by)\rb^{-1},
\end{split}
\end{equation}
 \begin{equation}\label{var_grad}
  \lV  \var\lb \nabla \hphi(\bx,\by)\rb^{-1}-  \var\lb \nabla \tphi(\bx,\by)\rb^{-1}\rV_\op=O\lp T(\bx-\by)\rp \;\mbox{as $|\bx-\by|\to\infty$}.
   \end{equation}
Note that
\[
\var\lb \tH(\bx, \by)\rb =\left[\begin{array}{cc}
\Var\lb  H(\bx)\rb & 0\\
0 & \var\lb H(\by)\rb
\end{array}
\right].
\]
Since $\Phi_\ga$ is stationary,  $\var\lb \tH(\bx, \by)\rb$ is \emph{independent} of $\bx$ and $\by$. We have
\[
\var\lb \hh(\bx, \by)\rb =\left[\begin{array}{cc}
\Var\lb  H(\bx)\rb & \cov\lb H(\bx),H(\by)\rb\\
& \\
\cov \lb H(\by), H(\bx)\rb & \var\lb H(\by)\rb
\end{array}
\right]
\]
\[
= \var\lb \tH(\bx, \by)\rb+ \underbrace{\left[\begin{array}{cc}
0 & \cov\lb H(\bx),H(\by)\rb\\
& \\
\cov \lb H(\by), H(\bx)\rb & 0
\end{array}
\right]}_{=: R_H(\bx,\by)}.
\]
We deduce  that as $\vert\bx-\by\vert\to\infty$ we have
 \begin{equation}\label{KR_cov_3}
 \lV  \var\lb \hh(\bx, \by)\rb-\var\lb \tH(\bx, \by)\rb\rV_\op=\Vert R_H(\bx,\by)\Vert_\op = O\lp T(\bx-\by)\rp.
 \end{equation}
 We denote by  $ \tH(\bx,y)^\flat$ the Gaussian random matrix 
 \[
  \tH(\bx,\by)^\flat= \tH(\bx,\by)-\bE\lb  \tH(\bx,\by)\cond   \nabla\tphi(\bx,\by\rb.
  \]
   Similarly, we denote by  $ \hh(\bx,\by)^\flat$ the Gaussian random matrix 
   \[
  \hh(\bx,y)^\flat=\hh(\bx,\by)-\bE\lb \hh(\bx,\by)\cond, \nabla\hPhi\rb.
  \]
 The distributions of   $ \tH(\bx,\by)^\flat$ and $ \hh(\bx,\by)^\flat$ are determined by the Gaussian  regression formula (\ref{reg2}).
 
 Since  $ \tH(\bx,y)$ and  $\nabla\tphi(\bx,\by)$  are independent we deduce 
 \[
 \Var\lb  \tH(\bx,\by)^\flat\rb= \Var\lb  \tH(\bx,\by)\rb.
 \]
 Using the regression formula we deduce  that for $\vert\bx-\by\vert>C_0$,
 \[
 \begin{split}
 \Var\lb \hh(\bx,\by)^\flat\rb=\var\lb \hh(\bx,\by)\rb\hspace{3cm}&\\
  -\cov\lb  \hh(\bx,\by), \nabla\hphi(\bx,\by)\rb \var\lb   \nabla\hphi(\bx,\by)\rb^{-1}\cov\lb   \nabla\hphi(\bx,\by),\hh(\bx,\by)\rb &|\\
  =\var\lb \tH(\bx,\by)^\flat\rb+R_H(\bx,\by)\hspace{5cm}&\\
  -\cov\lb  \hh(\bx,\by), \nabla\hphi(\bx,\by)\rb \var\lb   \nabla\hphi(\bx,\by)\rb^{-1}\cov\lb  \nabla\hphi(\bx,\by),\hh(\bx,\by)\rb. &
  \end{split}
  \]
We have
\[
\cov\lb  \hh(\bx,\by), \nabla\hphi(\bx,\by)\rb=\left[
\begin{array}{cc}
\cov\lb H(\bx),\nabla\Phi_\ga(\bx)\rb &  \cov\lb H(\bx),\nabla\Phi_\ga(\by)\rb\\
&\\
\cov\lb H(\by),\nabla\Phi_\ga(\bx)\rb &  \cov\lb H(\by),\nabla\Phi_\ga(\by)\rb
\end{array}
\right]
\]
\[
= \left[
\begin{array}{cc}
\cov\lb0 &  \cov\lb H(\bx),\nabla\Phi_\ga(\by)\rb\\
&\\
\cov\lb H(\by),\nabla\Phi_\ga(\bx)\rb &  0
\end{array}
\right].
\]
This implies
\[
\cov\lb  \hh(\bx,\by), \nabla\hphi(\bx,\by)\rb=O\lp T(\bx-\by)\rp\;\;\mbox{as $\vert\bx-\by\vert\to\infty$}.
\]
 We deduce  from (\ref{var_grad})  that
\[
\var\lb   \nabla\hphi(\bx,\by)\rb^{-1}= \var\lb   \nabla\tphi(\bx,\by)\rb^{-1} +O\lp T(\bx-\by)\rp.
\]
Since $\var\lb\nabla \tphi(\bx,\by)\rb$ is independent of $\bx$ and $\by$ we conclude that
\begin{equation}\label{KRc_cov4}
\cov\lb  \hh(\bx,\by), \nabla\hphi(\bx,\by)\rb \var\lb   \nabla\hphi(\bx,\by)\rb^{-1}\cov\lb  \nabla\hphi(\bx,\by),\hh(\bx,\by)\rb=O\lp T(\bx-\by)\rp,
\end{equation}
Since $\var\lb  \tH(\bx,\by)\rb$  is independent of $\bx,\by$   we deduce that there exists $\mu_0>0$ such that
\[
\var\lb  \tH(\bx,\by)^\flat\rb\geq \mu_0\one,\;\;\forall \bx\neq \by.
\]
We  deduce from (\ref{KRc_cov4}) and Lemma  \ref{lemma: cont_gauss_int}   that
\begin{equation}\label{KR_cov_5}
\Lv \bE\lb |\det \hh(\bx,\by)^\flat|\rb- \bE\lb |\det \tH(\bx,\by)^\flat|\rb\Rv= O\lp T(\bx-\by)^{1/2}\rp.
\end{equation}
Using (\ref{var_grad}) we deduce that as $|\bx-\by|\to\infty$ we have
\begin{equation}\label{KR_cov_5a}
\begin{split}
\Lv p_{ \nabla\hphi(\bx,\by)}(0)- p_{\nabla\tphi(\bx,\by)}(0)\rv&\\
=(2\pi)^{-m/2}\Lv \det \var\lb   \nabla\hphi(\bx,\by)\rb^{-1}-\det\var\lb   \nabla\tphi(\bx,\by)\rb^{-1}\Rv&=O\lp  T(\bx-\by)\rp.
\end{split}
\end{equation}
Note also that (\ref{KR_cov_3}) implies that 
\begin{equation}\label{sup_h}
\sup_{ \vert\bx- \by\vert_\infty>1}\Vert \var\lb \hh(\bx,\by)^\flat\rb\Vert_\op <\infty
\end{equation}
We can now estimate the right-hand-side of (\ref{KR_cov_2_per}). Using (\ref{KR_cov_5}), (\ref{KR_cov_5a}) and  (\ref{sup_h}) 
 we conclude that
 \begin{subequations}
 \begin{equation}
 \label{tildeKR}
\forall \vert\bx-\by\vert>1,\;\;\lv \Delta(\bx,\by)\rv=  O\lp \lv\bx-\by\rv^{-\infty}\rp,\;\;\mbox{as $|\bx-\by|\to\infty$}. 
 \end{equation}
 \begin{equation}\label{tildeKRa}
 \sup_{|\bx-\by|>1}\lv \Delta (\bx-\by)\rv<\infty.
 \end{equation}
 \end{subequations}
 
  \subsection{Conclusion}  Suppose that 
  \[
  \supp f\subset\{ |\bx|\leq_0\}.
  \]
   Denote by $\widehat{\eX}$ the  radial-blowup of $\bR^m\times \bR^m$ along the diagonal.  It is diffeomorphic to the product  $\bR^m \times S^{m-1}\times [0,\infty)$.

  Choose new orthogonal  coordinates $(\xi,\eta)$  given  by 
  \[
  \xi=\bx+\by,\;\;\eta=\bx-\by\,\Llra\, \bx=\frac{1}{2}(\xi+\eta),\;\;\by=\frac{1}{2}(\xi-\eta)
  \]
  then
  \[
  |\bx-\by|=|\eta|,\;\;d\bx d\by=2^{-2m} d\xi d\eta.
  \]
  Recall that $\hrho$ depends only on $\eta$. Note that if $\bx,\by\in \supp f$, then  $|\bx|,|\by|<r_0$ and  thus
  \begin{equation}\label{suppfR}
 \bx,\by\in \supp f\,\Ra\, |\xi|,\;|\eta|<\frac{1}{2}|\xi+\eta|+\frac{1}{2}|\xi-\eta|=|\bx|+|\by|\leq 2r_0.
  \end{equation}
  The natural projection $\pi:\widehat{\eX}\to\bR^m\times \bR^m$ can  given the explicit  description
  \[
  \bR^m \times S^{m-1}\times [0,\infty)\ni (\xi, \bnu,r)\mapsto (\xi,\eta)=(\xi, r\bnu)\in\bR^m\times \bR^m.
  \]
  We set
  \[
 w(\bx,\by)= |\bx-\by|^{m-2} \hrho(\bx,\by).
  \]
 
  We deduce from  \cite[Sec. 4.2]{BMM22} or  \cite[Appendix A.1]{EL}
  \begin{equation}\label{sup_wR}
 \sup_{0<|\bx-\by|\leq 1} \lv w(\bx,\by)\rv <\infty.
 \end{equation}
  It is easy to see that $\trho\circ \pi$ admits a continuous extension to the blow-up. Using (\ref{tildeKR}), (\ref{tildeKRa}) and (\ref{sup_wR}) we deduce that for any $p>0$ there exists a constant $K_p>0$, such that
  \begin{equation}\label{sup_deR}
  |x-y|^{m-1} \lv \Delta(\bx,\by)\rv \leq K_p\lp 1+|x-y|\rp^{-p+m-1},\;\;\forall \bx\neq \by
  \end{equation}
  Set
  \[
 \delta(\xi,\eta)= \Delta\lp \pi(\xi,\eta)\rp
 \]
Since $\Delta(\bx,\by)$ depends only on $\by-\bx$ we deduce that $\delta(\xi,\eta)$ is independent of $\xi$.   We have
  \[
 I(R)= \int_{\eX}\Delta(\bx,\by) f_R^{\boxtimes 2}(\bx,\by) d\bx d\by 
  \]
  \[
   =\frac{1}{2^{2m}}\int_{\bR^m}\int_{ |\bnu|=1,\,r\in(0,\infty)}r^{m-1}\delta \lp  \xi,r\bnu\rp f _R\Lp  \frac{\xi+r\bnu}{2}\Rp f_R\Lp\frac{\xi-r\bnu}{2}\Rp dr\vol_{S^{m-1}}[d\bnu] d\xi
  \]
  ($\xi=2R\zeta$)
  \[
 \stackrel{(\ref{suppfR})}{=}R^m\underbrace{ \int_{|\zeta\leq 2r_0}2^{-m}\left(\int_{\substack{|\bnu|=1\\
r>0}}r^{m-1} \delta \lp 0 ,r\bnu\rp f \Lp  \zeta+\frac{r\bnu}{2R}\Rp f(\zeta-\frac{r\bnu}{2R}\Rp dr\vol_{S^{m-1}}[d\bnu]\right)d\zeta}_{=:J(R)}.
 \]
 We deduce  from (\ref{suppfR}) and (\ref{sup_wR}) that  for any $p>0$  there exists $K_p>0$ such that for any $R>0$ ,  $|\bnu|=1$ and $r>0$  we have
 \[
 \Lv r^{m-1} \delta \lp 0 ,r\bnu\rp f \Lp  \zeta+\frac{r\bnu}{2R}\Rp f\Lp \zeta-\frac{r\bnu}{2R}\Rp \Rv\leq K_p\Vert f\Vert^2\lp 1+r\rp^{-p+m-1}.
 \]
  For $p>m$ we have
 \[
  \int_{|\zeta|\leq 2r_0}\left(\int_{(0,\infty\times S^{m-1}} \lp 1+r\rp^{-p+m-1} dr \vol_{S^{m-1}}[d\bnu]\right)d\zeta<\infty.
  \]
  The dominated  convergence theorem implies  that $J(R)$ has a finite limit  as $R\to\infty$.  More precisely
  \[
  \lim_{R\to\infty}J(R)= \int_{|\zeta|\leq2 r_0}\underbrace{\left(2^{-m}\int_{\substack{|\bnu|=1\\
r>0}}r^{m-1} \delta  \lp 0 ,r\bnu\rp  dr\vol_{S^{m-1}}[d\bnu]\right)}_{=:Z_m(\ga)}\; f(\zeta)^2 d\zeta.
\]
This concludes the proof of Theorem \ref{th: var_asy} (ii) with $V_m(\ga)= Z_m(\ga)+C_m(\ga)$.

   \appendix
   
   \section{Random measures}\label{s: ergodic}  The fact that $\Phi_\ga$ is a stationary random function implies  that the random measure $\fC_\ga$  is stationary. This  in itself has deep consequences. Let us elaborate.

 Denote by $\widehat{\Meas}(\bR^m)$ \index{$\widehat{\Meas}(\bR^m)$}the space  locally finite of Borel measures\index{measure! locally finite}  $\mu$ on $\bR^m$,  i.e.,  $\mu\lb B\rb<\infty$ for any bounded  Borel  subset $B\subset \bR^m$.  Each $f\in C^0_\cpt(\bR^m)$  defines a  map
  \[
  I_f:  \widehat{\Meas}(\bR^m)\to \bR,\;\; \mu\mapsto I_f(\mu)=\mu\lb f\rb:=\int_{\bR^m} f(\bx)\mu\lb dx\rb.
  \]
  The  \emph{vague topology} \index{vague! topology} on  $\widehat{\Meas}(\bR^m)$ is the smallest topology such that all the functions $I_f$, $f\in C_\cpt^0(\bR^m)$ are continuous.  As shown in \cite[Thm. 4.2]{Kalle_RM}, this topology is Polish, i.e., it is induced  by a complete  and separable  metric. We denote by $\eM$ is metric space.  T  We denote by $\Prob\lp \eM\rp$ the space of Borel  probability measures on $\eM$.
  
  A sequence $(\mu_n)$  in $\eM$ converges vaguely to $\mu \in \eM$, and we indicate  this as $\mu_n\stackrel{v}{\to}\mu$,  if and only if
  \[
  \mu_n\lb f\rb \to \mu\lb f\rb,\;\;\forall f\in C_\cpt^0(\bR^m).
  \]
  A random locally finite  measure \index{random measure! locally finite} on $\bR^m$ is a Borel measurable map
  \[
 \fM: \lp\Omega,\eS,\bP\rp\to \eM.
 \]
 Its distribution is a Borel probability measure $\bP_{\fM}$ on $\eM$.    
  
 Recall that a sequence of  probability measures $\mu_n\in \Prob(\eM)$ is said to converge weakly  to $\mu\in \Prob(\eM)$, and we indicate this $\mu_n\to\mu$  if
 \[
 \lim_{n\to \infty} \int_{\eX}Fd\mu_n=\int_{\eX} F d\mu,
 \]
  for any bounded and continuous function $F:\eM\to\bR$.  A sequence of random measures $\fM_n$ is said to converge weakly to the random measure  $\fM$ if the distributions $\bP_{\fM_n}$ converge weakly in  $\Prob(\eM)$ to $\bP_{\fM}$. We  use the notation $\fM_n\to\fM$ to indicate this.   
 
 A subset $Q\subset \bR^m$  is  a \emph{quasi-box} \index{quasi-box}  if it is  a product  of finite intervals
  \[
  Q=I_1\times \cdots \times I_m.
  \]
  The intervals $I_k$ need not be closed and could have length zero.  A quasi-box $Q$ is  called a box if all the intervals $I_k$ are closed and have nonzero lengths.

  We have the following result, \cite[Prop.11.1.VIII]{DaVere2}, \cite[Thm. 4.11]{Kalle_RM}.

  \begin{theorem}  Consider a  sequence $\lp \fM_n\rp_{n\in\bN}$  of random locally finite measures $\Prob\lp \eM\rp$. The following are equivalent.
  
  \begin{enumerate}
  
  \item  The sequence $\fM_n$ converges weakly to the  random locally finite measure  $\fM$.
  
  \item For any  $f\in C^0_\cpt(\bR^m)$, the random variables $\fM_n\lb f\rb$ converge in distribution to $\fM\lb f\rb$.
  
  \item For any quasi-box $Q\subset \bR^m$  the random variables $\fM_n\lb Q\rb$  converge in distribution to $\fM\lb Q \rb$.
  
  \item For any bounded Borel subset $S\subset \bR^m$ random variables $\fM_n\lb S\rb$  converge in distribution to $\fM\lb S \rb$
  \end{enumerate}
  \qed
  \end{theorem}

  There are other   modes of convergence  of random measures corresponding to the various modes of convergence of random  variables.   Suppose that
  \[
  \fM_n,\;\fM:\lp \Omega,\eS,\bP\rp\to \Prob\lp \eX\rp,\;\;n\in \bN
  \]
  are random   locally finite measures. We say that $\fM_n$ converges almost surely to $\fM$ and we indicate this $\fM_n\stackrel{\as}{\to} \fM$ if  there exists a $\bP$-negligible  set $\eN\in \Omega$ such that
  \[
  \fM_n(\omega) \stackrel{v}{\ra} \fM(\omega),\;\;\forall \omega\in \Omega\setminus \eN,
  \]
   i.e., 
  \[
  \fM_n\stackrel{\as}{\Lra}\fM\,\Llra\,  \fM_n\lb f\rb \stackrel{\as}{\Lra}\fM\lb f\rb,\;\forall f\in C^0_\cpt(\bR^m),.
  \]
  The convergence $\fM_n\stackrel{L^p}{\Lra}\fM$ is  defined in a similar  fashion namely 
   \[
  \fM_n\stackrel{L^p}{\Lra}\fM\,\Llra\,  \fM_n\lb f\rb \stackrel{L^p}{\Lra}\fM\lb f\rb,\;\forall f\in C^0_\cpt(\bR^m),.
  \]
One can show (see \cite[Lemma 4.8]{Kalle_RM}) that
\begin{subequations}
\begin{equation}\label{rand_meas_a}
 \fM_n\stackrel{\as}{\Lra}\fM\,\Llra\,  \fM_n\lb S\rb \stackrel{\as}{\Lra}\fM\lb S \rb,\;\;\mbox{for any bounded Borel set  $S\subset \bR^m$},
 \end{equation}
 \begin{equation}\label{rand_meas_b}
 \fM_n\stackrel{L^p}{\Lra}\fM\,\Llra\,  \fM_n\lb S\rb \stackrel{L^p}{\Lra}\fM\lb S \rb,\;\;\mbox{for any bounded Borel set  $S\subset \bR^m$}.
 \end{equation}
 \end{subequations}
   
  The action of $\bR^m$ on itself by translations  induces an action  on  $\eM$.   We denote by $\eI$ the sigma-subalgebra  of $\eB_{\eM}$ consisting  of  translation invariant Borel subsets of $\eM$.  A measure $\bP\in \Prob\lp \eM\rp$ is called stationary if its invariant  with respect to this action.
  A random measure  $\fM$ is called stationary  \index{random measure! stationary}  if its distribution $\bP_{\fM}$ is stationary.
  
  As discussed in \cite[Chap.12]{DaVere2} or \cite[Chap.5]{Kalle_RM}, every stationary  random locally finite   measure $\fM$ on $\bR^m$ has an asymptotic  intensity $\hat{\fM}\in L^1\lp \eM, \bP_\eM)$. More precisely
   \[
\widehat{\fM}:=\bE\lb \fM[C_1]\,\lv \eI\rb,\;\;C_1=[0,1]^m.
\]
This is an  integrable random variable $\widehat{\fM}$.    It has an ergodic interpretation. Wiener's ergodic theorem  shows  that  for  any compact convex subset $C\subset \bR^m$ containing the origin in the interior we have
\[ 
\widehat{\fM}= \lim_{N\to\infty}\frac{1}{N^m\vol\lb C\rb}\int_{NC}\fM[C_1-\bx\rb d\bx \;\;\as.
\]
 Thus, if the action of $\bR^m$ on $(\eM,\bP_{\fM})$ is ergodic,  then $\widehat{\fM}$ is $\bP_{\fM}$-$\as$ constant.

The intensity $\widehat{\fM}$ has another  ergodic interpretation; see \cite[Thm. 12.2.IV]{DaVere2} or  \cite[Th. 5.23]{Kalle_RM}. More precisely, for $C$ as above we have
\[
\frac{1}{N^m  \vol\lb C\rb} \fM\lb NC\rb\to \widehat{\fM}
\]
$\as$ and $L^1$. Moreover, if $\fM[C_1]\in L^p$, then the convergence holds also in $L^p$. This fact admits the following generalization.

A sequence $\vfi_N\in C^0_\cpt(\bR^m)$, $N\in \bN$ is called \emph{asymptotically} stationary if there exists $C>0$ such that
\[
\vfi_N\geq 0,\;\;\int_{\bR^m}\vfi_N(\bx) d\bx=C,\;\;\forall N,
\]
and 
\[
\lim_{N\to\infty}\int_{\bR^m}\lv \vfi_N(\bx)-\vfi_N(\bx-\by)\rv d\bx =0,\;\;\forall \by\in \bR^m.
\]
We have the following  result, \cite[Thm. 5.24]{Kalle_RM} 
\begin{theorem}\label{th: asy_sta} If  $(\vfi_N)_{N\in \bN}$ is asymptotically stationary, then
\[
 \fM\lb \vfi_N\rb\to C \widehat{\fM},
 \]
 in $L^1$ and $\as$.\qed
 \end{theorem}

\end{document}